\documentclass[12pt]{article}
\usepackage{amsmath,amssymb,amsthm}
\usepackage{comment}
\usepackage{color}
\numberwithin{equation}{section}
\newtheorem{thm}{Theorem}[section]
\newtheorem{prop}[thm]{Proposition}

\newtheorem{exam}[thm]{Example}
\newtheorem{rem}[thm]{Remark}
\newtheorem{lem}[thm]{Lemma}

\newtheorem{assum}[thm]{Assumption}

\def\Capa{\mathop{\rm Cap}}

\def\F{{\cal F}}

\def\R{{\mathbb R}}
\def\Pp{{\mathbb P}}
\def\Ee{{\mathbb E}}
\def\d{{\rm d}}

\def\address#1#2{\begingroup
\noindent\parbox[t]{16cm}{%
\small{\scshape\ignorespaces#1}\par\vskip1ex
\noindent\small{\itshape E-mail address}%
\/: #2\par\vskip4ex}\hfill%
\endgroup}%

\title{Martingale Nature and Laws of the Iterated Logarithm for
Markov Processes of Pure-Jump Type}
\author{Yuichi Shiozawa\quad and \quad Jian Wang{\thanks{Corresponding author}}}

\begin{document}
\maketitle
\begin{abstract}
We present sufficient conditions, in terms of the jumping kernels,
for two large classes of conservative Markov processes
of pure-jump type to be
purely discontinuous martingales
with finite second moment.
As an application, we establish the law of the iterated logarithm
for sample paths of the associated processes.
\end{abstract}
 \vskip 1cm

\allowdisplaybreaks

\section{Introduction}

It is well known that
any symmetric L\'evy process with finite first moment
possesses the martingale property
because of the independent increments property.
Apart from L\'evy processes,
the martingale property was studied
for a one-dimensional diffusion process with natural scale
(see \cite{K06, SN18} and references therein).
Note that this process is a time changed Brownian motion
and thus possesses the local martingale property
(see, e.g., \cite[Proposition
V.1.5]{RY99}). In \cite{K06, SN18}, a necessary and sufficient condition is given for this process  to be a martingale by adopting the Feller theory.

To the best of our knowledge,
except for
these  Markov processes mentioned above,
answers are not available in the literature to the following \lq\lq fundamental\rq\rq question
---{\it when does a Markov process become a martingale{\rm ?}}
The aim of this paper is to present explicit sufficient conditions
for two large classes of jump processes to be
purely discontinuous martingales with finite second moment
in terms of jumping kernels.
As an application,
we show Khintchine's law of the iterated logarithm (LIL)
for two classes of non-symmetric jump processes.
We also provide examples of non-symmetric jump processes
which are
purely discontinuous martingales with finite second moment
and satisfy the LIL.

To derive the martingale property for a jump process,
we apply two different approaches.
One is based on the infinitesimal generator along with the
moments calculus of the process,
and the other relies on the
componentwise decomposition of the process with
the aid of the semimartingale theory
(\cite[Chapter II, Section 2]{JS03}).
The assumptions of our paper are mild.
For example, condition \eqref{eq:second-moment} (or \eqref{jjj}) means
the existence of the second moment for the jumping kernel
(which seems to be necessary for the LIL),
while condition \eqref{eq:zero-drift} (or \eqref{e:non-drift}) roughly
indicates that there is no drift arising from jumps of the process.

Our motivation lies in the fact that
the LIL holds for
L\'evy processes with zero mean
and finite second moment
as proved by Gnedenko \cite{Gn43} (see also \cite[Proposition 48.9]{Sa13}).
J.-G. Wang \cite{W93} established this kind of result
for locally square integrable martingales and
obtained Gnedenko's result as a corollary (\cite[Corollary 2]{W93}).
For a symmetric jump process generated by non-local Dirichlet form,
we provided in \cite{SW19+} a sufficient
condition, in terms of the jumping kernel, for the long time behavior of the sample path
being
similar to that of the Brownian motion.
This condition implies the existence of the second moment for the jumping kernel.
Our approach in \cite{SW19+} was based on the long time heat kernel estimate.
Recently, it is proved in \cite{BKKL19+} that for
a special symmetric jump process,
the second moment condition on the jumping kernel is equivalent
to the validity of the LIL.
The approach of \cite{BKKL19+} is based on
the two-sided heat kernel estimate for full times.
See \cite{KKW, SW17} for related discussions on this topic.
In contrast with
\cite{BKKL19+, KKW, SW17, SW19+},
our result is applicable to non-symmetric jump processes.
Moreover, our approach is elementary in the sense that
we use the martingale theory of stochastic processes.

\ \

The rest of this paper is organized as follows.
In Section \ref{section2}, we first consider the martingale property of a class of Feller processes of pure-jump type,
and then prove the LIL.
Some new examples including jump processes of variable order are also presented.
The corresponding discussions for non-symmetric Hunt processes
generated by lower bounded semi-Dirichlet forms of pure-jump type are considered in Section \ref{section3}.

Throughout this paper, the letters $c$ and $C$ (with subscript)
denote finite positive constants which may vary from place to place.
For $x\in {\mathbb R}^d$, let $x^{(i)}$ be its $i$th coordinate;
that is, $x=(x^{(1)},\dots, x^{(d)})\in {\mathbb R}^d$.
We denote by $\langle \cdot, \cdot \rangle$ the standard inner product on $\R^d$.
Let ${\cal B}(\R^d)$ and ${\cal B}_b(\R^d)$ denote, respectively,
the family of Borel measurable sets on $\R^d$
and the set of bounded Borel measurable functions on $\R^d$.
Let $C_c^{\infty}(\R^d)$
(resp.\ $C_c^2(\R^d)$)
be the set of smooth
(resp.\ twice continuously differentiable)
functions with compact support in $\R^d$, and let $C_b^2(\R^d)$ be the set of twice continuously differentiable functions on $\R^d$
with all bounded derivatives.
Let $C_{\infty}(\R^d)$ be the set of continuous functions on $\R^d$ vanishing at infinity.

\section{Martingale Nature for Feller processes}\label{section2}
\subsection{Preliminaries}\label{subsect:prelim}

Let
$X:=\{\Omega, \F,(\F_t)_{t\geq 0}, (X_t)_{t\geq 0}, (\Pp_x)_{x\in \R^d}\}$
be a time-homogeneous Markov process on $\R^d$.
Let $(T_t)_{t\geq 0}$ be
the Markov semigroup associated with the process $X$, i.e.,
$$
T_tu(x)=\Ee_x u(X_t),\quad u\in {\cal B}_b(\R^d),
(t,x)\in [0,\infty)\times \R^d.
$$
According to \cite[Definition 1.16]{BSW},
we call $X$ a Feller process if
$(T_t)_{t\geq 0}$ is a Feller semigroup on $C_\infty(\R^d)$; that is,
it satisfies the following properties:
\begin{itemize}
\item (Feller property)
for any $u\in C_{\infty}(\R^d)$ and  $t>0$, $T_tu\in C_{\infty}(\R^d)$;
\item (strong continuity)
for any $u\in C_{\infty}(\R^d)$, $\|T_tu-u\|_{\infty}\rightarrow 0$ as $t\rightarrow 0$.
\end{itemize}

In what follows, we suppose that $X$ is a Feller process on $\R^d$.
Let $(\R^d)_{\Delta}:=\R^d\cup\{\Delta\}$ be a one-point compactification of $\R^d$.
Then by \cite[Theorems 1.19 and 1.20]{BSW},
we may and do assume that $X$ satisfies the next properties:
\begin{itemize}
\item $X$ has a c\`{a}dl\`{a}g modification; that is, for every $x\in \R^d$,
a map $t\mapsto X_t(\omega)$ is
right continuous with left limits in $(\R^d)_{\Delta}$
for $\Pp_x$-a.s.\ $\omega\in \Omega$;
\item the filtration $(\F_t)_{t\geq 0}$ is complete and right continuous,
and $X$ is a strong Markov process with this filtration.
\end{itemize}

Define
$$D(L)=\left\{u\in C_{\infty}(\R^d)\, \Big|\,
\text{$\displaystyle \lim_{t\rightarrow 0}\frac{T_tu-u}{t}$ exists in $C_{\infty}(\R^d)$}\right\}$$
and
$$Lu= \lim_{t\rightarrow 0}\frac{T_tu-u}{t}, \quad u\in D(L).$$
The pair $(L,D(L))$ is called a Feller generator of the Feller semigroup $(T_t)_{t\geq 0}$.
If $C_c^{\infty}(\R^d)\subset D(L)$, then
the general form of $L$
is known (see, e.g.,  \cite[Theorem 2.21]{BSW}),
and $X$ enjoys
an analogous L\'evy-It\^{o} decomposition
(see, e.g., \cite[Theorem 3.5]{Sch} or \cite[Theorem 2.44]{BSW}).
Furthermore, by  \cite[Theorem 1.36]{BSW}, for any $u\in D(L)$,
\begin{align}\label{eq:martingale-u}
M_t^{[u]}:=u(X_t)-u(X_0)-\int_0^t Lu(X_s)\,\d s, \quad  t\geq 0,
\end{align}
is a martingale with respect to $({\F}_t)_{t\geq 0}$.

Throughout this section,
we impose the following conditions on the Feller generator $(L,D(L))$.
\begin{assum}\label{assum:generator}\it
Let $(L,D(L))$ be a Feller generator of the Feller semigroup associated with $X$
so that the next two conditions are satisfied:
\begin{enumerate}
\item[{\rm(i)}] $C_c^{\infty}(\R^d)\subset D(L)${\rm ;}
\item[{\rm(ii)}] for any $u\in C_c^{\infty}(\R^d)$,
\begin{align}\label{eq:generator}
Lu(x)=\int_{\R^d\setminus\{0\}}\left(u(x+z)-u(x)-\langle \nabla u(x), z\rangle{\bf 1}_{\{|z|<1\}}(z)\right)\,N(x,\d z),
\end{align}
where for any fixed $x\in \R^d$,
$N(x,\d z)$ is a non-negative deterministic measure on
$\R^d\setminus\{0\}$ such that
\begin{equation}\label{eq:second-moment}
\sup_{x\in \R^d}\int_{\R^d\setminus\{0\}} |z|^2
\, N(x,\d z)<\infty.
\end{equation}
\end{enumerate}
\end{assum}
We comment
on Assumption \ref{assum:generator}.
Since the generator $L$
in \eqref{eq:generator} consists of the jump part only,
Assumption \ref{assum:generator} implies that $X$ is a semimartingale of pure jump-type;
that is, there is no continuous part
in the semimartingale decomposition of $X$
(i.e., no diffusion term involved) (see, e.g., \cite[Theorem 2.44]{BSW}).
The kernel $N(x,\d z)$ is called the jumping kernel of $X$.
Note that by \eqref{eq:second-moment},
\begin{equation}\label{eq:first-moment}
\sup_{x\in \R^d}\int_{\{|z|\geq 1\}}|z|\,N(x,\d z)<\infty
\end{equation}
and so
\begin{equation}\label{eq:first-moment--}\sup_{x\in \R^d}\int_{\{|z|\geq 1\}}|z^{(i)}| \,N(x,\d z)<\infty, \quad 1\leq i\leq d.\end{equation}

Under the full conditions of Assumption \ref{assum:generator},
$X$ is conservative,
i.e.,
$T_t 1=1$ for any $t\geq 0$
(\cite[Theorem 2.33]{BSW}).
According to \cite[Theorem 2.37 c) and a)]{BSW},
$C_c^2(\R^d)\subset D(L)$
and the operator $(L,C_c^{\infty}(\R^d))$
has a unique extension to $C_b^2(\R^d)$,
which is still denoted by $(L,C_b^2(\R^d))$,
such that
the representation \eqref{eq:generator} remains true for this extension.
We see further by
\cite[Theorems 2.37 i) and 1.36]{BSW}
that
for any $u\in C_b^2(\R^d)$,
\eqref{eq:martingale-u} is also
a martingale with respect to $({\cal F}_t)_{t\geq 0}$.

\subsection{Martingale property of Feller processes}\label{Section2.2}
In this subsection,  we present a sufficient condition
on the jumping kernel $N(x,\d z)$ such that $X$ is a
purely discontinuous
martingale with finite second moment.

\begin{thm}\label{thm:markov-mtg}
Let Assumption {\rm \ref{assum:generator}} hold.
Assume also that
for any $x\in \R^d$,
\begin{align}\label{eq:zero-drift}
\int_{\{|z|\geq 1\}}z^{(i)} N(x,\d z)=0, \quad
1\leq i\leq d.
\end{align}
Then
$X$ is a
purely discontinuous martingale such that
for each $t>0$ and $i=1,\dots, d$, $X_t^{(i)}$ has finite second moment and
the
predictable quadratic variation
of $X$
is given by
\begin{align}\label{eq:cov}
\langle X^{(i)}, X^{(j)}\rangle_t
=\int_0^t \left(\int_{\R^d\setminus\{0\}}z^{(i)}z^{(j)}N(X_s,\d z)\right) \,\d s ,\quad 1\leq i,j\leq d,\,t>0.
\end{align}
\end{thm}

In the following, we will show
two different approaches to prove Theorem \ref{thm:markov-mtg}.
The first one relies on the expression of the generator $L$ via moment calculus,
and the second one is based on the canonical representation of the semimartingale.

\subsubsection{Generator approach}\label{section2.1.1}

The key  ingredient of the generator approach to establish Theorem \ref{thm:markov-mtg}
is the following statement for
the first and second moments of $X_t$.
\begin{prop}\label{lem:moment-formula}
Let Assumption {\rm \ref{assum:generator}} hold. Then, for  any $t>0$ and $i=1,\dots, d$,
$X_t^{(i)}$ has finite second moment, and, for any $x_0\in \R^d$,
\begin{equation}\label{eq:n-moment}
\begin{split}
\Ee_{x_0}\left[(X_t^{(i)}-x_0^{(i)})^2\right]
&=\Ee_{x_0}\left[\int_0^t\left(\int_{\R^d\setminus\{0\}}(z^{(i)})^2\,N(X_s,\d z)\right)\,\d s \right]\\
&\quad+2\Ee_{x_0}\left[\int_0^t(X_s^{(i)}-x_0^{(i)})\left(\int_{\{|z|\geq 1\}}z^{(i)}\,N(X_s,\d z)\right)\d s\right]
<+\infty.
\end{split}
\end{equation} Moreover, it also holds that
\begin{equation}\label{eq:n-moment-1}
\Ee_{x_0}\left[X_t^{(i)}-x_0^{(i)}\right]
=\Ee_{x_0}\left[\int_0^t\left(\int_{\{|z|\geq 1\}}z^{(i)}N(X_s,\d z)\right)\,d s\right].
\end{equation}
\end{prop}

We postpone the proof of Proposition \ref{lem:moment-formula}
to the end of this part.
Using this proposition, we can present the

\begin{proof}[Proof of Theorem {\rm \ref{thm:markov-mtg}}]
Let Assumption \ref{assum:generator} and \eqref{eq:zero-drift} hold.
We first show that $X$ is a
martingale
having finite second moment and satisfying
\eqref{eq:cov}.
That $X$ has finite second moment has been claimed in Proposition \ref{lem:moment-formula}.
According to \eqref{eq:n-moment-1} and \eqref{eq:zero-drift},
\begin{equation*}
\Ee_{x_0}[X_t^{(i)}-x_0^{(i)}]
=\Ee_{x_0}\left[\int_0^t\left(\int_{\{|z|\geq 1\}}z^{(i)}N(X_s,\d z)\right)\,d s\right]
=0,\quad t>0,
1\le i\le d.
\end{equation*}
Then, by the Markov property,
for any $0<s\le t$,
\begin{equation*}
\Ee_{x_0}[X_t^{(i)} \mid {\cal F}_s]
=\Ee_{X_s}[X_{t-s}^{(i)}]=X_s^{(i)},
\end{equation*}
whence $(X_t^{(i)})_{t\geq 0}$ is a martingale
for all $1\le i\le d. $

Since
\begin{equation*}
\Ee_{x_0}[(X_t^{(i)}-x_0^{(i)})^2]
=\Ee_{x_0}[(X_t^{(i)})^2]-(x_0^{(i)})^2,
\end{equation*}
we see by
\eqref{eq:n-moment}, \eqref{eq:zero-drift} and
the Markov property that
\begin{equation*}
L_t:=(X_t^{(i)})^2-\int_0^t\int_{\R^d\setminus\{0\}}(z^{(i)})^2\,N(X_s,\d z)\,\d s, \quad  t\geq 0
\end{equation*}
is a martingale.
This implies that
\begin{equation*}
\langle X^{(i)} \rangle_t=\int_0^t\int_{\R^d\setminus\{0\}}(z^{(i)})^2\,N(X_s,\d z)\,\d s
\end{equation*}
and thus we obtain \eqref{eq:cov}.

We next show that the martingale $X$ is purely discontinuous.
Let $X^{(i),c}$ and $X^{(i),d}$ be the continuous and purely discontinuous parts of $X^{(i)}$,
respectively (see \cite[Theorem I.4.18]{JS03}).
Let
$\Delta X_s^{(i)}=X_s^{(i)}-X_{s-}^{(i)}$ for all $s>0$.
Denote by $([X^{(i)}]_t)_{t\ge0}$ the quadratic variation of $X^{(i)}$.
Then, by \cite[Theorem I.4.52]{JS03},
\begin{equation}\label{eq:quad-var}
[X^{(i)}]_t=\langle X^{(i),c} \rangle_t+\sum_{s\in (0,t]: \Delta X_s\ne 0}(\Delta X_s^{(i)})^2.
\end{equation}
Since $\Ee_{x_0}\left[[X^{(i)}]_t\right]=\Ee_{x_0}\left[\langle X^{(i)} \rangle_t\right]$
and the L\'evy system formula (see, e.g., \cite[Remark 2.46]{BSW} and references therein) implies that
\begin{equation*}
\Ee_{x_0}\left[\langle X^{(i)} \rangle_t\right]
=\Ee_{x_0}\left[\int_0^t\left(\int_{\R^d\setminus\{0\}}(z^{(i)})^2\,N(X_s,\d z)\right)\,\d s\right]
=\Ee_{x_0}\left[\sum_{s\in (0,t]: \Delta X_s\ne 0}(\Delta X_s^{(i)})^2\right],
\end{equation*}
we see
from \eqref{eq:quad-var} that
$\Ee_{x_0}\left[\langle X^{(i),c} \rangle_t\right]=0$ for any $t>0$.
Therefore, $\Pp_{x_0}(\langle X^{(i),c} \rangle_t=0)=1$ for any $t>0$.
Noticing that $(\langle X^{(i),c} \rangle_t)_{t\ge 0}$ is right continuous,
we further obtain
\begin{equation*}
\Pp_{x_0}\left(\text{$\langle X^{(i),c} \rangle_t=0$ for any $t>0$}\right)=1,
\end{equation*}
and so, by \cite[Lemma I.4.13 a)]{JS03},
\begin{equation*}
\Pp_{x_0}\left(\text{$X_t^{(i),c}=0$ for any $t>0$}\right)=1.
\end{equation*}
Hence the martingale $X$ is purely discontinuous.
\end{proof}

Next, we are  back to the proof of Proposition \ref{lem:moment-formula}.
Let $X$ be a Feller process on $\R^d$ such that
the associated Feller generator $(L,D(L))$ fulfills Assumption {\rm \ref{assum:generator}}.
For any differentiable function $u$ on $\R^d$ such that
\begin{equation*}
\int_{\{0<|z|<1\}}|u(x+z)-u(x)-\langle\nabla u(x),z\rangle|\,N(x,\d z)
+\int_{\{|z|\geq 1\}}|u(x+z)-u(x)|\,N(x,\d z)<\infty
\end{equation*}
for any $x\in \R^d$, we will
define $Lu$ by \eqref{eq:generator}
in the pointwise sense.
This definition is clearly consistent with $(L,C_b^2(\R^d))$.
Moreover,
under \eqref{eq:second-moment},
$Lu$ can be well
defined by \eqref{eq:generator}
even
for unbounded twice continuously differentiable functions $u$;
for example, $u(x)=x^{(i)}-x^{(i)}_0$ and $u(x)=(x^{(i)}-x_0^{(i)})^2$ for any $x_0\in \R^d$.

To show Proposition \ref{lem:moment-formula}, we
start with the following simple lemma.
For any $x_0\in \R^d$ and any $l\in {\mathbb N}$,
let $\phi_l(x)\in C_c^\infty(\R^d)$
satisfy
\begin{equation}\label{e:function-1}
\phi_l(x)
\begin{cases}
=1,&\quad 0\leq |x-x_0|\leq l,\\
\in [0,1],&\quad l<|x-x_0|<l+1,\\
=0, &\quad |x-x_0|\geq l+1.
\end{cases}
\end{equation}

\begin{lem}\label{lem:lim-generator}
Under Assumption {\rm \ref{assum:generator}},
the following two statements hold.
\begin{itemize}
\item[{\rm(1)}]
The function $f(x)=(x^{(i)}-x_0^{(i)})^2$ for any fixed $x_0\in \R^d$ satisfies
\begin{equation}\label{eq:lim-general}
\lim_{l\rightarrow\infty}L(f\phi_l)(x)=Lf(x),\quad x\in \R^d,
\end{equation}
where $\phi_l$ is defined by \eqref{e:function-1}.
Moreover, there is a constant $C_1>0$ such that for all $l\in {\mathbb N}$ and $x_0,x\in \R^d$,
\begin{equation}\label{eq:lim-general-1}
|L(f\phi_l)(x)|\le   C_1(1+f(x)),
\end{equation}
and so $$|Lf(x)|\le C_1(1+f(x)).$$

\item[{\rm(2)}]
The function $f(x)=x^{(i)}-x_0^{(i)}$ for any fixed $x_0\in \R^d$ also satisfies \eqref{eq:lim-general},
and  there is a constant $C_2>0$ such that for all $x_0, x\in \R^d$,
$$|L(f\phi_l)(x)|\le   C_2(1+|f(x)|)$$
and
$$|Lf(x)|\le C_2(1+|f(x)|).$$
\end{itemize}
\end{lem}

\begin{proof}
We only prove (1), since (2) can be verified similarly.
Since $f\phi_l\in C_c^2(\R^d)\subset D(L)$,
\begin{equation}\label{eq:separate}
\begin{split}
L(f\phi_l)(x)
&=\int_{\{0<|z|<1\}}(f\phi_l(x+z)-f\phi_l(x)-\langle \nabla(f\phi_l)(x),z\rangle)\,N(x,\d z)\\
&\quad +\int_{\{|z|\geq 1\}}(f\phi_l(x+z)-f\phi_l(x))\,N(x,\d z)\\
&=:{\rm (I)}+{\rm (II)}.
\end{split}
\end{equation}
For any $x,z\in \R^d$ with $0<|z|<1$,
\begin{equation*}
\begin{split}
&f\phi_l(x+z)-f\phi_l(x)-\langle \nabla(f\phi_l)(x),z\rangle\\
&=f(x)(\phi_l(x+z)-\phi_l(x)-\langle \nabla\phi_l(x),z\rangle)+\phi_l(x+z)(f(x+z)-f(x)-\langle \nabla f(x), z\rangle)\\
&\quad
+(\phi_l(x+z)-\phi_l(x))\langle \nabla f(x), z\rangle.
\end{split}
\end{equation*}
Note that, by the Taylor theorem,
there exist positive constants $C_*$ and $C_{**}$ such that
for any $l\in {\mathbb N}$ and $x,z\in \R^d$,
$$|\phi_l(x+z)-\phi_l(x)|\leq C_*|z|,\quad  |\phi_l(x+z)-\phi_l(x)-\langle \nabla \phi_l(x),z\rangle|\leq C_{**}|z|^2.$$
This yields that
there exists $c_1>0$ such that for any $l\in {\mathbb N}$
and $x, z\in \R^d$,
\begin{equation*}
|f\phi_l(x+z)-f\phi_l(x)-\langle \nabla(f\phi_l)(x),z\rangle|
\leq c_1(1+f(x))(z^{(i)})^2.
\end{equation*}
Hence, by \eqref{eq:second-moment},
\begin{equation}\label{eq:i-bound}
\begin{split}
|{\rm (I)}|\leq
c_1(1+f(x))\int_{\{0<|z|<1\}}(z^{(i)})^2\,N(x,\d z)
\leq c_2(1+f(x)).
\end{split}
\end{equation}
Because
\begin{equation*}
f\phi_l(x+z)-f\phi_l(x)-\langle \nabla(f\phi_l)(x),z\rangle
\rightarrow f(x+z)-f(x)-\langle \nabla f(x),z\rangle\quad \hbox{ as } l\rightarrow\infty,
\end{equation*}
we get, by the dominated convergence theorem,
that
\begin{equation}\label{eq:i-conv}
{\rm (I)}\rightarrow
\int_{\{0<|z|<1\}}(f(x+z)-f(x)-\langle \nabla f(x),z\rangle)\,N(x,\d z)\quad \hbox{ as } l\rightarrow\infty.
\end{equation}

Since there exists $c_3>0$ such that for any $l\in {\mathbb N}$ and $x,z\in \R^d$,
\begin{equation*}
|f\phi_l(x+z)-f\phi_l(x)|
\leq f(x+z)+f(x)
\leq c_3(f(x)+(z^{(i)})^2),
\end{equation*}
 it also follows from \eqref{eq:second-moment} that
  for any $x\in \R^d$,
\begin{equation}\label{eq:i-bound-1}
 |{\rm (II)}|\leq c_3\left(f(x)\int_{\{|z|\geq 1\}}\,N(x,\d z)
+\int_{\{|z|\geq 1\}}(z^{(i)})^2\,N(x,\d z)\right)
\leq c_4(1+f(x)).
\end{equation}
With this at hand, we get by the dominated convergence theorem again that
\begin{equation*}
{\rm (II)}\rightarrow \int_{\{|z|\geq 1\}}(f(x+z)-f(x))\,N(x,\d z)\quad \hbox{ as } l\rightarrow\infty.
\end{equation*}
Combining this with \eqref{eq:i-conv}, we complete the proof of \eqref{eq:lim-general}.
Furthermore, \eqref{eq:lim-general-1} follows from \eqref{eq:i-bound} and \eqref{eq:i-bound-1}.
\end{proof}

Using Lemma \ref{lem:lim-generator},
we can now
present the

\begin{proof}[Proof of Proposition {\rm \ref{lem:moment-formula}}]
Throughout the proof, we fix $x_0\in \R^d$.
We first show that $X_t^{(i)}$ has finite second moment and \eqref{eq:n-moment} holds.
Let $f(x)=(x^{(i)}-x_0^{(i)})^2$, and $\phi_l$ be the function defined by \eqref{e:function-1}.
Then, $f\phi_l\in C_c^2(\R^d)\subset D(L)$, and so
\begin{equation*}
M_t^{[f\phi_l]}=f\phi_l(X_t)-f\phi_l(X_0)-\int_0^tL(f\phi_l)(X_s)\,\d s, \quad t\geq 0
\end{equation*}
is a martingale as mentioned in \eqref{eq:martingale-u}
and remarks before Subsection \ref{Section2.2}.
Hence
\begin{equation}\label{eq:mtg-exp}
\Ee_{x_0}[f\phi_l(X_t)]=f\phi_l(x_0)+\Ee_{x_0}\left[\int_0^tL(f\phi_l)(X_s)\,\d s\right]
=\Ee_{x_0}\left[\int_0^tL(f\phi_l)(X_s)\,\d s\right].
\end{equation}

For $m\in {\mathbb N}$, let $\tau_m:=\inf\{t>0: |X_t-x_0|\geq m\}$.
Then, by the optional stopping theorem, for all $l,m\in {\mathbb N}$ and $t>0$,
\begin{equation}\label{eq:opt-1}
\Ee_{x_0}[f\phi_l(X_{t\wedge \tau_m})]
=\Ee_{x_0}\left[\int_0^{t\wedge \tau_m}L(f\phi_l)(X_s)\,\d s\right].
\end{equation}

According to \eqref{eq:lim-general-1},
there is $c_1>0$ such that for all $l,m\in {\mathbb N}$ and $t>0$,
\begin{equation}\label{eq:bound}
\begin{split}
\Ee_{x_0}\left[\int_0^{t\wedge \tau_m}|L(f\phi_l)(X_s)|
\,\d s\right]
\le & c_1\Ee_{x_0}\left[\int_0^{t\wedge \tau_m}[1+f(X_s)]\,\d s\right]\\
=& c_1\Ee_{x_0}\left[\int_0^{t\wedge \tau_m}[1+f(X_{s-})]\,\d s\right]\\
=&c_1\Ee_{x_0}\left[\int_0^{t\wedge \tau_m}[1+(m^2\wedge f)(X_{s-})]\,\d s\right]\\
=& c_1\Ee_{x_0}\left[\int_0^{t\wedge \tau_m}[1+(m^2\wedge f)(X_{s})]\,\d s\right]\\
\le &c_1\Ee_{x_0}\left[\int_0^{t}[1+(m^2\wedge f)(X_{s\wedge \tau_m})]\,\d s\right].
\end{split}
\end{equation}
Here, $X_{t-}:=\lim_{s\rightarrow t-0}X_s$ for any $t>0$,
in the first and the third equalities above
we used the fact that almost surely $t\mapsto X_t$ is a c\`{a}dl\`{a}g function which has
at most countably many jumps on $[0,t]$, i.e., $\{s\in[0,t]:X_s\neq X_{s-}\}$  is almost surely
a set of Lebesgue measure zero;
in the second equality we used the fact that $f(X_{s-})\le m^2$ for all $s\in [0, \tau_m]$;
and the last inequality follows from
the monotonicity of the Lebesgue integral.

On the other hand, the monotone convergence theorem yields that for all $m\in {\mathbb N}$ and $t>0,$
$$\lim_{l\to\infty}\Ee_{x_0}[f\phi_l(X_{t\wedge \tau_m})]\ge \lim_{l\to\infty}\Ee_{x_0}[(m^2\wedge f)\phi_l(X_{t\wedge \tau_m})]= \Ee_{x_0} [(m^2\wedge f)(X_{t\wedge \tau_m})].$$
Combining this and \eqref{eq:bound} with \eqref{eq:opt-1},
we have for all $m\in{\mathbb N}$ and $t>0$,
$$\Ee_{x_0} [1+(m^2\wedge f)(X_{t\wedge \tau_m})]\le 1+ c_1\int_0^{t}\Ee_{x_0}[1+(m^2\wedge f)(X_{s\wedge \tau_m})]\,\d s.$$
Then, by the Gronwall inequality,
$$ \Ee_{x_0} [1+(m^2\wedge f)(X_{t\wedge \tau_m})]\le e^{c_1t}.$$
Since the process $X$ is conservative, letting $m\to\infty$,
we have
for all $t>0$,
\begin{equation}\label{e:bound-f}
\Ee_{x_0} [f(X_{t})]\le e^{c_1t}-1.
\end{equation}
Therefore, $X_t^{(i)}$ has finite second moment.

Furthermore, by \eqref{e:bound-f}, Lemma \ref{lem:lim-generator}
and the dominated convergence theorem,
letting $l\to\infty$ in \eqref{eq:mtg-exp}, we obtain that for all $t>0$,
$$\Ee_{x_0}[f(X_{t})]=\Ee_{x_0}\left[\int_0^{t}Lf(X_s)\,\d s\right].$$
Note that, for any $x\in \R^d$,
\begin{align*}
Lf(x)=
&\int_{\{0<|z|<1\}}\left(f(x+z)-f(x)-\langle \nabla f(x), z\rangle\right)\,N(x,\d z)\\
&+\int_{\{|z|\ge1\}}\left(f(x+z)-f(x)\right)\,N(x,\d z)\\
=&\int_{\{0<|z|<1\}} (z^{(i)})^2\,N(x,\d z)+ \int_{\{|z|\ge1\}}\left((z^{(i)})^2+2(x^{(i)}-x_0^{(i)})z^{(i)}\right)\,N(x,\d z)\\
=&\int_{\R^d\setminus\{0\}} (z^{(i)})^2\,N(x,\d z)+ 2(x^{(i)}-x_0^{(i)})\int_{\{|z|\ge1\}}z^{(i)}\,N(x,\d z),
\end{align*}
where all the integrals above are well defined by \eqref{eq:second-moment} and \eqref{eq:first-moment}.
With this at hand, we obtain
\eqref{eq:n-moment}.

We next show \eqref{eq:n-moment-1}.
We note that $\Ee_{x_0}[|X_t^{(i)}|]<\infty$ because $X_t^{(i)}$ has finite second moment.
Let
$f(x)=x^{(i)}-x_0^{(i)}$. Then, \eqref{eq:mtg-exp} still holds true.
Hence,
following the argument above
and using Lemma \ref{lem:lim-generator}(2) and the dominated convergence theorem,
we can also
obtain \eqref{eq:n-moment-1}.
\end{proof}

\subsubsection{Canonical representation approach}\label{subsect:ooo}
In this part,
we present another approach to Theorem \ref{thm:markov-mtg},
which is based on the canonical representation of the semimartingale
(see \cite[Section 2.5]{BSW} and \cite[Chapter II, Section 2]{JS03}).

\begin{proof}[Proof of Theorem {\rm \ref{thm:markov-mtg}}]
Let $X$ satisfy Assumption \ref{assum:generator} and \eqref{eq:zero-drift}.
Let $\delta_{(s,w)}$ be the Dirac measure at $(s,w)\in (0,\infty)\times (\R^d\setminus \{0\})$
and $\mu^X(\d t,\d z)$ a (random) jumping measure of $X$ defined by
\begin{equation*}
\mu^X(\d t,\d z)
=\sum_{s\in (0,\infty): \Delta X_s\ne 0}\delta_{(s,\Delta X_s)}(\d t, \d z).\end{equation*}
If $\nu^X(\d t,\d z)$ denotes the compensator (or the dual predictable projection) of $\mu^X$,
then by the L\'evy system formula,
\begin{equation}\label{eq:dual-proj}
\nu^X(\d t,\d z)= N(X_{t-},\d z)\,\d t.
\end{equation}
Let
$W_1^{(i)}(z)=z^{(i)}{\bf1}_{\{|z|< 1\}}(z)$ for $z\in \R^d$.
Then by \cite[Theorem 2.44]{BSW} and \cite[Theorems II.2.34 and II.2.42]{JS03},
$X$ is a semimartingale and has the following componentwise decomposition:
\begin{align}\label{eq:semi-mtg}
X_t^{(i)}
=X_0^{(i)}+W_1^{(i)}*(\mu^X-\nu^X)_t
+\sum_{s\in (0,t]: \Delta X_s\geq 1}\Delta X_s^{(i)},
\quad 1\leq i\leq d.
\end{align}
Here $W_1^{(i)}*(\mu^X-\nu^X)$ is a stochastic integration; that is, it is
a locally square integrable and purely discontinuous martingale such that
$$\langle W_1^{(i)}*(\mu^X-\nu^X)\rangle_t
=\int_0^t \left(\int_{\{0<|z|<1\}}(z^{(i)})^2N(X_s,\d z)\right)\,\d s$$
(see, e.g.,
\cite[Definition II.1.27 and its subsequent comment,
and Theorem II.1.33]{JS03}
for the definition and properties of stochastic integrals
with respect to a random measure).
Then, by \cite[Proposition 4.50]{JS03} and \eqref{eq:second-moment},
we have for any $t\geq 0$,
\begin{equation*}
\begin{split}
&\Ee_{x_0}\left[[W_1^{(i)}*(\mu^X-\nu^X)]_t\right]
=\Ee_{x_0}\left[\langle W_1^{(i)}*(\mu^X-\nu^X)\rangle_t\right]\\
&=\Ee_{x_0}\left[\int_0^t \left(\int_{\{0<|z|<1\}}(z^{(i)})^2N(X_s,\d z)\right)\,\d s\right]
\leq t\sup_{x\in \R^d}\int_{\{0<|z|<1\}}(z^{(i)})^2N(x,\d z)<\infty.
\end{split}
\end{equation*}
Hence, according to \cite[Corollary II.6.3]{P04},
\begin{equation}\label{e:note-mar}(W_1^{(i)}*(\mu^X-\nu^X)_t)_{t>0}\hbox{ is a martingale with finite second moment}.
\end{equation}

Let $W_2^{(i)}(z)=z^{(i)}{\bf 1}_{\{|z|\ge1\}}(z)$.
Then by \eqref{eq:second-moment},
$W_2^{(i)}*(\mu^X-\nu^X)$ is a locally square integrable and purely discontinuous martingale.
Furthermore, since \eqref{eq:zero-drift} yields that
$$
W_2^{(i)}*\nu^X_t=\int_0^t\left(\int_{\{|z|\geq 1\}}z^{(i)} N(X_s,\d z)\right)\, \d s=0,
$$
we obtain, by \cite[Proposition II.1.28]{JS03}
and \eqref{eq:first-moment--},
$$
W_2^{(i)}*(\mu^X-\nu^X)
_t
=W_2^{(i)}*\mu^X
_t
=\sum_{s\in (0,t]: \Delta X_s\geq 1}\Delta X_s^{(i)}.
$$
Hence, by \eqref{eq:semi-mtg},
$X$ is a locally square integrable and purely discontinuous martingale
such that
\begin{equation}\label{e:note-mar2}
X^{(i)}_t=X_0^{(i)}+(W_1^{(i)}+W_2^{(i)})*(\mu^X-\nu^X)_t
\end{equation}
and
$$\langle X^{(i)}\rangle_t=\int_0^t \left(\int_{\R^d\setminus\{0\}}(z^{(i)})^2N(X_s,\d z)\right)\, \d s.$$
Following the argument for \eqref{e:note-mar}, we can further claim that
$(X^{(i)}_t)_{t\ge0}$ is a martingale with finite second moment.

Similarly, we have
$$\langle X^{(i)}\pm X^{(j)}\rangle_t=\int_0^t \left(\int_{\R^d\setminus\{0\}}(z^{(i)}\pm z^{(j)})^2N(X_s,\d z)\right)\, \d s,$$
which implies that
$$
\langle X^{(i)}, X^{(j)}\rangle_t
=\frac{1}{4}\left(\langle X^{(i)}+X^{(j)}\rangle_t-\langle X^{(i)}-X^{(j)}\rangle_t\right)
=\int_0^t\left(\int_{\R^d\setminus\{0\}}z^{(i)}z^{(j)}N(X_s,\d z)\right)\,\d s.
$$
Therefore, we obtain \eqref{eq:cov}.  \end{proof}

\subsection{Application: law of the iterated logarithm}\label{subsect:LIL}
Let $X$ be a Feller process
satisfying Assumption \ref{assum:generator}
and \eqref{eq:zero-drift}.
Then $X$ is a conservative Markov process,
and is a
purely discontinuous martingale
with finite second moment
by Theorem \ref{thm:markov-mtg}.
Hence, for any unit vector ${\mathbf r}=(r^{(1)},\dots, r^{(d)})\in {\mathbb R}^d$,
$X^{\mathbf r}:=(\langle X_t,{\mathbf r}\rangle)_{t\ge0}$ is
also a
purely discontinuous martingale
with finite second moment
such that for any $t\geq 0$,
\begin{equation}\label{e:mmm}
\begin{split}
\langle X^{\mathbf r}\rangle_t
&=\sum_{i=1}^d (r^{(i)})^2\langle X^{(i)}\rangle_t
+2\sum_{1\leq i<j\leq d}r^{(i)}r^{(j)} \langle X^{(i)}, X^{(j)}\rangle_t \\
&=\int_0^t\left(\int_{\R^d\setminus\{0\}}\langle {\mathbf r}, z\rangle^2 N(X_s,\d z)\right)\,\d s,
\end{split}
\end{equation}
where the last equality follows from Theorem \ref{thm:markov-mtg}.

In this subsection, we establish Khintchine's law of the iterated logarithm for $X$.
To do so,  we
make use of the stochastic integral representation of $X$.
As in Subsection \ref{subsect:ooo},
let $\mu^X(\d t,\d z)$ be the (random) jumping measure of $X$
and $\nu^X(\d t,\d z)$ the compensator of $\mu^X$
given by \eqref{eq:dual-proj}.
Let
$V^{(i)}(z)=z^{(i)}$ for $z\in \R^d$.
Then,
by \eqref{e:note-mar2},
\begin{align*}
X_t^{(i)}=X_0^{(i)}+V^{(i)}*(\mu^X-\nu^X)_t, \quad 1\leq i\leq d.
\end{align*}
Here $V^{(i)}*(\mu^X-\nu^X)$ is a stochastic integration;
moreover, by Theorem \ref{thm:markov-mtg},
it is defined as a
locally square
integrable and purely discontinuous martingale
such that
\begin{equation*}
\langle V^{(i)}*(\mu^X-\nu^X)\rangle_t
=\int_0^t \left(\int_{\R^d\setminus\{0\}}(z^{(i)})^2N(X_{s-},\d z)\right)\,\d s,\quad t>0.
\end{equation*}

We will further impose the following assumption on the jumping kernel $N(x,\d z)$.
\begin{assum}\label{assum:bound}
The jumping kernel $N(x,\d z)$ satisfies the next two conditions{\rm :}
\begin{enumerate}
\item[{\rm (i)}]
there exist non-negative measures $\nu_1(\d z)$ and $\nu_2(\d z)$ on $\R^d\setminus\{0\}$ such that for $1\le i\le d$,
$$\int_{\R^d\setminus\{0\}} (z^{(i)})^2 \,\nu_1(\d z)>0,\quad \int_{\R^d\setminus\{0\}} |z|^2 \, \nu_2(\d z)<\infty,$$
and for any $x\in \R^d$ and
$A\in {\cal B}(\R^d\setminus\{0\})$,
$$\nu_1(A)\leq N(x,A)\leq \nu_2(A);
$$
\item[{\rm (ii)}]
for $x\in \R^d$, let
$$a_{ij}(x)=\int_{\R^d\setminus\{0\}}z^{(i)}z^{(j)}N(x,\d z),\quad 1\leq i, j\leq d.$$
Then there exist constants
 $0<\lambda\leq \Lambda<\infty$ such that for any
$x,\xi\in \R^d$,
$$\lambda|\xi|^2\leq \sum_{i,j=1}^da_{ij}(x)\xi^{(i)}\xi^{(j)}\leq \Lambda|\xi|^2.$$
\end{enumerate}
\end{assum}

Note that
\begin{equation*}
\begin{split}
\sum_{i,j=1}^da_{ij}(x)\xi^{(i)}\xi^{(j)}
=\int_{\R^d\setminus\{0\}}\langle \xi, z \rangle^2\,N(x,\d z)
&\leq |\xi|^2\int_{\R^d\setminus\{0\}}|z|^2\,N(x,\d z)\\
&\leq |\xi|^2\sup_{x\in \R^d}\int_{\R^d\setminus\{0\}}|z|^2\,N(x,\d z),
\end{split}
\end{equation*}
so it holds that
$$\Lambda\leq \sup_{x\in \R^d}\int_{\R^d\setminus\{0\}}|z|^2\,N(x,\d z).$$

\begin{thm}\label{thm:lil}
Let $X$ be a Feller process such that
Assumptions {\rm \ref{assum:generator}} and {\rm \ref{assum:bound}},
and \eqref{eq:zero-drift} hold.
\begin{enumerate}
\item[{\rm (1)}]
For every $x\in {\mathbb R}^d$ and every unit vector ${\mathbf  r}\in {\mathbb  R}^d$,
\begin{equation*}
\Pp_x \left( \limsup_{t\to \infty}
\frac{X_t^{{\mathbf r}}}
{\sqrt{2\langle X^{{\mathbf r}}\rangle_t \log\log \langle X^{{\mathbf r}}\rangle _t}}=1\right)=1.
\end{equation*}
\item[{\rm (2)}]
For every $x\in \R^d$,
$$\Pp_x \left(\sqrt{\lambda}\leq  \limsup_{t\to \infty}
\frac{|X_t|}{\sqrt{2t\log\log t}}\le \sqrt{\Lambda}\right)=1.$$
\end{enumerate}
\end{thm}

\begin{proof}
Let $X$ be a Feller process satisfying
Assumptions {\rm \ref{assum:generator}} and {\rm \ref{assum:bound}},
and \eqref{eq:zero-drift}.
Then, by Theorem \ref{thm:markov-mtg},
$X$ is a
purely discontinuous martingale with finite second moment.

We first prove (1) by applying \cite[Theorem 1]{W93} to $X$.
To do so, we see that $X$ satisfies Assumption A (i) and (ii) of \cite{W93}.
Let
$$F(x):=\int_{\R^d\setminus\{0\}}|z|^2\,N(x,\d z)$$
and
$$c_1:=\int_{\R^d\setminus\{0\}}|z|^2\,\nu_1(\d z), \quad c_2:=\int_{\R^d\setminus\{0\}}|z|^2\,\nu_2(\d z).$$
Then, by Assumption \ref{assum:bound} (i),
$c_1$ and $c_2$ are finite positive constants and
\begin{equation}\label{eq:f-bound}
c_1\leq  F(x)\leq c_2, \quad  x\in \R^d.
\end{equation}
Hence if we define
$$C_t:=\sum_{i=1}^d\langle X^{(i)}\rangle_t
=\int_0^t F(X_s)\,{\rm d}s
=\int_0^t F(X_{s-})\,{\rm d}s,\quad t\ge0,$$
then $(C_t)_{t\geq 0}$ is a predictable increasing process such that
$$\Pp_x\left(\lim_{t\rightarrow\infty}C_t=\infty\right)=1.$$
Let
\begin{equation*}
N_t({\rm d}z):=\frac{1}{F(X_{t-})}\,N(X_{t-}, {\rm d}z).
\end{equation*}
Then, by \eqref{eq:dual-proj},
\begin{equation}\label{eq:decomp}
\nu^X(\d t, \d z)=N_t(\d z)\,\d C_t.
\end{equation}

Define a $d\times d$-symmetric matrix $S_t$ by
$$S_t:=\frac{1}{F(X_{t-})}(a_{ij}(X_{t-}))_{1\leq i,j\leq d}.$$
Note that each entry of $S_t$ is a predictable density function
of the predictable quadratic variation of $X$ in \eqref{eq:cov}
with respect to $C_t$.
If we let
$\Lambda_1:=\Lambda/c_1$ and $\lambda_1:=\lambda/c_2$,
then Assumption \ref{assum:bound} implies that
$\Lambda_1I-S_t$ and $S_t-\lambda_1I$ are nonnegative definite matrices,
where $I$ is a $d\times d$-unit matrix.
Combining this with \eqref{eq:decomp},
we have verified Assumption A (i) of \cite{W93} for $X$.

By \eqref{eq:f-bound}, we have for any $t>0$ and $A\in {\cal B}(\R^d\setminus\{0\})$,
$$N_t(A)\leq \frac{1}{c_1}\nu_2(A).$$
This implies that $X$ satisfies also Assumption A (ii) of \cite{W93}.

Let ${\bf r}$ be a unit vector in $\R^d$.
Since \eqref{e:mmm} and Assumption \ref{assum:bound} (ii)  yield that
\begin{equation}\label{eq:qv-lower}
\langle X^{{\bf r}}\rangle_t
=\int_0^t\left(\int_{\R^d\setminus\{0\}}\langle {\mathbf r},z\rangle^2\, N(X_s,\d z)\right)\,\d s
\geq \lambda t,
\end{equation}
we obtain
\begin{equation*}
\Pp_x\left(\lim_{t\rightarrow\infty}\langle X^{{\mathbf r}}\rangle_t=\infty\right)=1.
\end{equation*}
Therefore, (1) follows by applying \cite[Theorem 1]{W93} to $X$.

We next prove (2) in the same way
as for the law of the iterated logarithm for the multidimensional Brownian motion
(see, e.g., \cite[Exercise 1.5.17]{St11}).
Let ${\bf r}$ be a unit vector in $\R^d$.
Since $|X_t|\geq |
X_t^{{\mathbf r}}|$, we have, by \eqref{eq:qv-lower},
for large $t$,
\begin{align*}
\frac{|X_t|}{\sqrt{2t\log\log t}}
\geq \frac{|
X_t^{{\mathbf r}}|}{\sqrt{2t\log\log t}}
&=\frac{\sqrt{2\langle X^{{\bf r}}\rangle_t\log\log \langle X^{{\bf r}}\rangle_t}}{\sqrt{2t\log\log t}}
\frac{|
X_t^{{\mathbf r}}|}{\sqrt{2\langle X^{{\bf r}}\rangle_t\log\log \langle X^{{\bf r}}\rangle_t}}\\
&\geq \sqrt{\lambda}\frac{\sqrt{\log\log (\lambda t)}}{\sqrt{\log\log t}}
\frac{|
X_t^{{\mathbf r}}|}{\sqrt{2\langle X^{{\bf r}}\rangle_t\log\log \langle X^{{\bf r}}\rangle_t}}.
\end{align*}
Hence, according to (1),  for every $x\in \R^d$,
\begin{equation}\label{e:lll}\Pp_x \left( \limsup_{t\to \infty}
\frac{|X_t|}{\sqrt{2t\log\log t}}\ge\sqrt{\lambda}\right)=1.\end{equation}

Since \eqref{e:mmm} and Assumption \ref{assum:bound} (ii) imply that
for any unit vector ${\mathbf r}\in \R^d$,
$$
\langle X^{{\mathbf r}}\rangle_t
=\int_0^t\left(\int_{\R^d\setminus\{0\}}\langle {\mathbf r},z\rangle^2\,N(X_s,\d z)\right)\,\d s
\leq \Lambda t,
$$
we have, by \eqref{eq:qv-lower} again,
for large $t$,
\begin{align*}
\frac{|X_t^{{\mathbf r}}|}{\sqrt{2t\log\log t}}
&=\frac{\sqrt{2\langle X^{{\mathbf r}}\rangle_t\log\log \langle X^{{\mathbf r}}\rangle_t}}{\sqrt{2t\log\log t}}
\frac{|X_t^{{\mathbf r}}|}{\sqrt{2\langle X^{{\mathbf r}}\rangle_t\log\log \langle X^{{\mathbf r}}\rangle_t}}\\
&\leq \frac{\sqrt{2\Lambda t\log\log (\Lambda t)}}{\sqrt{2t\log\log t}}
\frac{|X_t^{{\mathbf r}}|}{\sqrt{2\langle X^{{\mathbf r}}\rangle_t\log\log \langle X^{{\mathbf r}}\rangle_t}}.
\end{align*}
Then, due to (1),
\begin{equation}\label{eq:lil-component}
\limsup_{t\rightarrow\infty}\frac{|X_t^{{\mathbf r}}|}{\sqrt{2t\log\log t}}\leq \sqrt{\Lambda},\quad a.s.
\end{equation}

Noting that $|X_t|\leq \sqrt{d}\max_{1\leq i\leq d}|X_t^{(i)}|$,
we obtain
from \eqref{e:lll} and \eqref{eq:lil-component} that
for almost surely there is a finite random variable
$t_0:=t_0(\omega)>e$ so that
$$\sup_{t\geq t_0
}\frac{|X_t|}{\sqrt{2t\log\log t}}\in (0,\infty)$$
being bounded by two positive (non-random) constants $C_1\le C_2$.
Let $\{{\bf r}_n
\}_{n\ge1}$ be a family of unit vectors in $\R^d$
forming a dense set in the unit sphere $S^{d-1}$.
Then, for any $\varepsilon\in (0,2C_2)$, there exists $l\geq 1$ such that
$$S^{d-1}\subset \bigcup_{k=1}^lB({\bf r}_k,\varepsilon/(2C_2)).$$
Hence if we let
$$R_t:=\frac{X_t}{\sqrt{2t\log\log t}},$$
then almost surely,
for any $\varepsilon>0$ and $t\ge t_0$
there exists a random variable $j:=j(\varepsilon,t, \omega)\in \{1,\dots, l\}$ such that
\begin{align*}
|R_t/|R_t|-\langle R_t/|R_t|, {\bf r}_j\rangle{\bf r}_j|
&\leq |R_t/|R_t|-{\bf r}_j|+|1-\langle R_t/|R_t|, {\bf r}_j\rangle|\\
&=|R_t/|R_t|-{\bf r}_j|+|R_t/|R_t|-{\bf r}_j|^2/2\leq \varepsilon/C_2,
\end{align*}
which
along with the fact that $|R_t|\le C_2$ for all $t\ge t_0$ a.s. shows that
$$|R_t-\langle R_t, {\bf r}_j\rangle{\bf r}_j|\leq \varepsilon \quad  \text{for any $\varepsilon>0$ and $t\geq t_0$, a.s.}$$
Since this inequality implies that
$$
|R_t|\leq \max_{1\leq k\leq l}|\langle R_t, {\bf r}_k\rangle|+\varepsilon
\quad  \text{for any $\varepsilon>0$ and $t\geq t_0$, a.s.},
$$
we have, by \eqref{eq:lil-component},
$$
\limsup_{t\rightarrow\infty}|R_t|
\leq\limsup_{t\rightarrow\infty}\max_{1\leq k\leq l}|\langle R_t, {\bf r}_k\rangle|+\varepsilon
\leq \sqrt{\Lambda}+\varepsilon \quad \text{for any $\varepsilon>0$, a.s.}
$$
Letting $\varepsilon\rightarrow  0$, we get
$$
\limsup_{t\rightarrow\infty}\frac{|X_t|}{\sqrt{2t\log\log t}}\leq \sqrt{\Lambda},\quad \text{a.s.}
$$
Hence the proof is complete.
\end{proof}

\begin{rem}\rm
We do not know the zero-one law for the tail events of $X$
(see \cite[Theorem 2.10]{KKW} and references therein
for symmetric jump processes with heat kernel estimates).
In particular, it is not clear
whether there exists a positive non-random constant $c_0$ such that
\begin{equation*}
\Pp_{x_0}\left(\limsup_{t\rightarrow \infty}\frac{|X_t|}{\sqrt{2t\log\log t}}=c_0\right)=1.
\end{equation*}
On the other hand, since
$$
\langle X^{(i)}\rangle_t\geq t\inf_{x\in \R^d}\int_{\R^d\setminus\{0\}}(z^{(i)})^2\,N(x,\d z),\quad 1\leq i\leq d, $$
it follows from the same argument as for the proof of {\rm (2)} that
$$
\Pp_x\left(\limsup_{t\rightarrow\infty}\frac{|X_t|}{\sqrt{2t\log\log t}}
\geq \max_{1\leq i\leq d}\left[\inf_{x\in \R^d}\int_{\R^d\setminus\{0\}}(z^{(i)})^2\,N(x,\d z)\right]\right)=1.
$$
\end{rem}

\subsection{Examples}\label{section00}
In this subsection, we provide a class of Feller processes
which are
purely discontinuous martingales with finite second moment,
and
satisfy Khintchine's law of the iterated logarithm.
For any $u\in C_c^{\infty}(\R^d)$, let
$$Lu(x)=L_0u(x)+Bu(x),$$ where
$$L_0u(x)
=\int_{\{0<|z|<1\}}\left(u(x+z)-u(x)-\langle \nabla u(x), z\rangle \right)
\frac{c(x)}{|z|^{d+\alpha(x)}}\,\d z,$$ and
$$Bu(x)=\int_{\{|z|\geq 1\}} \left(u(x+z)-u(x) \right) n_0(x,z)\,\d z.$$ Here,
$c(x)$ and $\alpha(x)$ are positive measurable functions on $\R^d$,
and $n_0(x,z)$ is a non-negative Borel measurable function on
$\R^d\times \{z\in \R^d:|z|\ge1\}$.

We will impose the following conditions on $\alpha(x)$, $c(x)$ and $n_0(x,z)$, respectively.
\begin{assum}\label{assum:index}
\begin{itemize}
\item[{\rm (i)}] The index function $\alpha(x)$ satisfies the next conditions:
\begin{itemize}
\item[{\rm (i-1)}] $0<\inf_{x\in \R^d}\alpha(x)\leq \sup_{x\in \R^d}\alpha(x)<2${\rm ;}
\item[{\rm (i-2)}] $\lim_{r\to0} |\log r|\left[\sup_{|x-y|\le r} |\alpha(x)-\alpha(y)|\right]=0;$
\item[{\rm (i-3)}] $\int_0^1 \sup_{|x-y|\le r}|\alpha(x)-\alpha(y)|r^{-1} \,\d r<\infty.$
\end{itemize}
\item [{\rm (ii)}] The coefficient $c(x)$ is continuous
and
$0<\inf_{x\in \R^d}c(x)\le \sup_{x\in \R^d}c(x)<\infty$.
\item[{\rm (iii)}] The function $n_0(x,z)$ satisfies the following conditions:
\begin{itemize}
\item[{\rm (iii-1)}] there is a non-negative Borel
measurable
function $\tilde n_0(x)$ on
$\{x\in \R^d:|x|\ge1\}$ such that
$$\int_{\{|z|\ge1\}}|z|^2\tilde n_0(z)\, \d z<\infty$$
and for any $x,z \in \R^d$ with $|z|\geq 1$,
$$n_0(x,z)\le \tilde n_0(z);$$
\item[{\rm (iii-2)}] for almost every
$z\in \R^d$ with $|z|\ge1$,
the function $x\mapsto n_0(x,z)$ is continuous on $\R^d${\rm ;}
\item[{\rm (iii-3)}]
for any $x\in \R^d$,
$$\int_{\{|z|\geq 1\}}z^{(i)} n_0(x,z)\,\d z=0, \quad  1\le i\le d.$$
\end{itemize}
\end{itemize}
\end{assum}

\begin{prop}\label{prop:feller}
Under
Assumption {\rm \ref{assum:index}},
$(L, C_c^{\infty}(\R^d))$ is closable on $C_{\infty}(\R^d)$
and its closure is the generator of a Feller semigroup.
Furthermore, let $X:=\{(X_t)_{t\geq 0}, (\Pp_x)_{x\in \R^d}\}$
be the Feller process on $\R^d$
associated with the closure of $(L, C_c^{\infty}(\R^d))$.
Then $X$ is a
purely discontinuous martingale with finite second moment,
and satisfies Khintchine's law of the iterated logarithm.
\end{prop}

To prove Proposition \ref{prop:feller}, we start from the following operator $(A, C_c^{\infty}(\R^d))$:
$$Au(x)=\int_{\R^d
\setminus\{0\}}
\left(u(x+z)-u(x)-\langle \nabla u(x), z\rangle{\bf 1}_{\{0<|z|<1\}}\right)
\frac{C_{\alpha(x)}}{|z|^{d+\alpha(x)}}\,\d z,\quad u\in C_c^{\infty}(\R^d),$$
where
$$C_{\alpha(x)}=\frac{\alpha(x)2^{\alpha(x)-1}\Gamma((\alpha(x)+d)/2)}{\pi^{d/2}\Gamma(1-\alpha(x)/2)}.$$
By \cite[Theorem 2.2]{B88}
(see also \cite[Theorem 3.31]{BSW} or \cite[Theorem 5.2]{K17}),
there exists a Feller process
$Y:=\{(Y_t)_{t\geq 0}, (\Pp_x)_{x\in \R^d}\}$
under Assumption \ref{assum:index} (i)
such that
\begin{itemize}
\item $Y$ is a unique solution to the $(A,C_c^{\infty}(\R^d))$-martingale problem;
\item $C_c^{\infty}(\R^d)$ is the core of the Feller generator of $Y$.
\end{itemize}
Indeed, as Bass \cite{B88} proved the existence and uniqueness of
the $(A,C_b^2(\R^d))$-martingale problem,
$Y$ is called the stable-like process in the sense of Bass in the literature.

Let $(A,{\cal D}(A))$ be the Feller generator of $Y$.
Note that, by Assumption \ref{assum:index} (ii),
 $m(x)=c(x)C_{\alpha(x)}^{-1}$
is a continuous function on $\R^d$
bounded from below and above by positive constants.
Hence the operator $(L_1,D(L_1))=(m(\cdot)A,D(A))$ is also a Feller generator (see, e.g., \cite[Theorem 4.1]{BSW}).

Next, we consider the following operator $(B_1,C_{\infty}(\R^d))$:
\begin{equation*}
B_1 u(x)=Bu(x)-\int_{\{|z|\geq 1\}}(u(x+z)-u(x))\frac{c(x)}{|z|^{d+\alpha(x)}}\, \d z.
\end{equation*}
Note that for $u\in C_c^{\infty}(\R^d)$,
\begin{equation*}
B_1u=Bu-(L_1-L_0)u.
\end{equation*}
We can claim that
\begin{lem}\label{lem:bdd-op}
Under Assumption {\rm \ref{assum:index}}, the
operator $B_1$ is a bounded linear operator on $C_{\infty}(\R^d)$.
\end{lem}

\begin{proof}
For $u\in C_{\infty}(\R^d)$, let
$$B_2u(x)=\int_{\{|z|\geq 1\}}(u(x+z)-u(x))\frac{c(x)}{|z|^{d+\alpha(x)}}\, \d z.$$
Then, we will prove that
both $B$ and $B_2$ are bounded linear operators on $C_{\infty}(\R^d)$.
If it holds, then we can prove the assertion.
For simplicity, we verify the conclusion only
for the operator $B$.

First, by Assumption
\ref{assum:index}
(iii-1),
there is a constant $C>0$ such that for any $u\in C_{\infty}(\R^d)$,
$\|Bu\|_{\infty}\leq C\|u\|_{\infty}$;
according to Assumption
\ref{assum:index}
(iii-1)-(iii-2) and the dominated convergence theorem,
the function $Bu$ is continuous on $\R^d$.
To complete the proof, it is enough  to show that $Bu(x)\rightarrow 0$ as $|x|\rightarrow\infty$.
For any $\varepsilon>0$, there exists a constant $R>1$ such that for any $x\in \R^d$ with $|x|\geq R$,
$|u(x)|\leq \varepsilon$.
By Assumption
\ref{assum:index}
(iii-1), we can also assume that
$$
\int_{\{|z|\ge R\}} n_0(x,z)\,\d z\le \varepsilon.
$$
Now, assume that $|x|\geq 2R$. We write
\begin{align*}
|Bu(x)|\le
& \int_{\{|z|\ge1\}}|u(x)|n_0(x,z)\,\d z+ \int_{\{|z|\ge1\}}|u(x+z)|n_0(x,z)\,\d z\\
\le
& \varepsilon \int_{\{|z|\ge1\}} n_0(x,z)\,\d z +\int_{\{|z|\ge1, |x+z|\ge R\}}|u(x+z)|n_0(x,z)\,\d z\\
&+ \int_{\{|z|\ge1, |x+z|< R\}}|u(x+z)|n_0(x,z)\,\d z \\
\le &2 \varepsilon  \sup_{x\in \R^d}\int_{\{|z|\ge1\}} n_0(x,z)\,\d z
+ \|u\|_\infty\int_{\{|z|\ge R\}} n_0(x,z)\,\d z\\
\le & \varepsilon \left(2\sup_{x\in \R^d}\int_{\{|z|\ge1\}} n_0(x,z)\,\d z+\|u\|_\infty\right),
\end{align*}
which yields the desired assertion.
\end{proof}

\begin{proof}[Proof of Proposition {\rm \ref{prop:feller}}]
Recall that $L_1=m(\cdot)A$ and $m(x)=c(x)C_{\alpha(x)}^{-1}$
is a continuous function on $\R^d$
bounded from below and above by positive constants.
Therefore,  $(L_1,C_c^{\infty}(\R^d))$ is closable on $C_{\infty}(\R^d)$.
Since $B_1$ is bounded on $C_{\infty}(\R^d)$ by Lemma {\rm \ref{lem:bdd-op}} and $L=L_1+B_1$,
$(L,C_c^{\infty}(\R^d))$ is also closable on $C_{\infty}(\R^d)$.
As $L$ satisfies the positive maximum principle,
the closure of $(L, C_c^{\infty}(\R^d))$ is the generator of a Feller semigroup
by {\rm \cite[Proposition 2.1]{SU10}}.

Let $X:=\{(X_t)_{t\geq 0}, (\Pp_x)_{x\in \R^d}\}$ be the Feller process on $\R^d$
associated with the closure of $(L, C_c^{\infty}(\R^d))$,
and let $N(x,\d z)$ be the jumping kernel of $X$.
Then $N(x,\d z)$ is absolutely continuous with respect to the Lebesgue measure,
and the density function $n(x,z)$ is given by
\begin{equation*}
n(x,z)=\frac{c(x)}{|z|^{d+\alpha(x)}}{\bf1}_{\{0<|z|<1\}}
+n_0(x,z){\bf1}_{\{|z|\geq 1\}}.
\end{equation*}
Since $N(x,\d z)=n(x,z)\,\d z$ fulfills
the condition  \eqref{eq:zero-drift}  and Assumption \ref{assum:bound}
by Assumption \ref{assum:index},
$X$ is a
purely discontinuous martingale with finite second moment,
and satisfies Khintchine's law of the iterated logarithm, respectively,
by
Theorems \ref{thm:markov-mtg} and \ref{thm:lil}.
\end{proof}

At the end of this
subsection,
we present two examples of $n_0(x,z)$ such that
Assumption
\ref{assum:index}
{\rm (iii)} is satisfied.
\begin{exam}
Let $c_0(x)$ be a continuous function on $\R^d$
which is bounded from below and above by positive constants.
\begin{enumerate}
\item[{\rm (1)}]
Let $\beta_1(x)$ be a positive continuous function on $\R^d$ such that $\inf_{x\in \R^d} \beta_1(x)>2$.
Then the function
$$n_0(x,z)=\frac{c_0(x)}{|z|^{d+\beta_1(x)}} $$
satisfies Assumption
{\rm \ref{assum:index}}
{\rm (iii)}.
\item[{\rm (2)}]
Let $\lambda$ be a positive constant
and $\beta_2(x)$ a positive continuous function on $\R^d$
such that $\inf_{x\in \R^d} \beta_2(x)>0$.
Then the function
$$n_0(x,z)=  c_0(x) e^{-\lambda |z|^{\beta_2(x)}}$$
satisfies Assumption
{\rm \ref{assum:index}}
{\rm (iii)}.
\end{enumerate}
\end{exam}

\section{Martingale nature and LIL of Hunt processes}\label{section3}

In this section, we discuss the martingale property
and Khintchine's  law of the iterated logarithm
for a class of jump-type Hunt processes generated
by regular lower bounded semi-Dirichlet forms.

\subsection{Regular lower bounded semi-Dirichlet forms}
\label{subsect:semi-dirichlet}
In this subsection,
we recall the notion of regular lower bounded semi-Dirichlet forms
by following \cite[Section 1]{O} and \cite{FU}.
Let $M$ be a locally compact separable metric space
and $\mu$ a positive Radon measure on $M$ with full support.
Let $(\eta,\F)$ be a bilinear form on $L^2(M;\mu)$,
and $\eta_{\alpha}(u,u)=\eta(u,u)+\alpha\|u\|_{L^2(M;\mu)}^2$
for $\alpha\geq 0$.
We say that $(\eta,\F)$ is a lower bounded closed form,
if there exists $\alpha_0\geq 0$ such that the next three conditions hold:
\begin{enumerate}
\item[(i)]
$\eta_{\alpha_0}(u,u)\geq 0$ for any $u\in \F$;
\item[(ii)]
there exists $K\geq 1$ such that for any $u,v\in \F$,
$|\eta(u,v)|\leq K\sqrt{\eta_{\alpha_0}(u,u)}\sqrt{\eta_{\alpha_0}(v,v)}$;
\item[(iii)]
$\F$ is complete with respect to
the norm $\|u\|_{\eta_{\alpha}}:=\sqrt{\eta_{\alpha}(u,u)}$
for some/any $\alpha>\alpha_0$.
\end{enumerate}

Let $(\eta,\F)$ be a lower bounded closed form on $L^2(M;\mu)$.
Then $(\eta,\F)$ is called Markovian, if for any $u\in \F$,
$Uu:=0\vee u\wedge 1\in {\cal F}$ and
$\eta(Uu,u-Uu)\geq 0$.
A lower bounded semi-Dirichlet form  is
by definition a lower bounded closed Markovian form.
For a lower bounded semi-Dirichlet form $(\eta,{\cal F})$,
there exists a strongly continuous Markovian semigroup $(T_t)_{t\geq 0}$ on $L^2(M;\mu)$
such that for any $\alpha>\alpha_0$,
the $\alpha$-resolvent
$G_{\alpha}f:=\int_0^{\infty}e^{-\alpha t}T_t f\,\d t$ satisfies
$\eta_{\alpha}(G_{\alpha}f,g)=\int_{M}fg\,\d \mu$
for any $f\in L^2(M;\mu)$ and $g\in {\cal F}$
(\cite[Theorems 1.1.2 and 1.1.5]{O}).
We can further extend $(T_t)_{t>0}$ and $(G_{\alpha})_{\alpha>\alpha_0}$
to $L^{\infty}(M;\mu)$ so that
$\|T_t\|_{\infty}\leq 1$ for all $t>0$ and $\|G_{\alpha}\|_{\infty}\leq 1/\alpha $ for all $\alpha>\alpha_0$
(see the discussion after the proof of \cite[Theorem 1.1.5]{O}).

Let $C_0(M)$ be the set of continuous functions on $M$ with compact support.
We say that a lower bounded semi-Dirichlet form $(\eta,{\cal F})$ is regular,
if $\F\cap C_0(M)$ is dense
both in $L^2(M;\mu)$
with respect to the norm $\|\cdot\|_{\eta_{\alpha}}$ for
some/any $\alpha>\alpha_0$
and in $C_0(M)$ with respect to the uniform norm.

Let $(\eta,\F)$ be a regular lower bounded semi-Dirichlet form on $L^2(M;\mu)$,
and fix a constant  $\alpha>\alpha_0$.
For an open set $O\subset E$,
let ${\cal L}_O:=\{u\in \F :
\text{$u\geq 1$ $\mu$-a.e.\ on $O$}\}$,
and define the capacity of $O$ by
\begin{equation*}
{\Capa}(O)=
\begin{cases}
\inf\left\{\eta_{\alpha}(u,u) :
 u\in {\cal L}_O\right\} & \text{if ${\cal L}_O\ne \emptyset$},\\
\infty & \text{if ${\cal L}_O=\emptyset$}.
\end{cases}
\end{equation*}
The capacity of a set $A\subset M$ is defined by
\begin{equation*}
{\Capa}(A)=\inf\{\Capa(O): \text{$O\subset M$ is open and $A\subset O$}\}.
\end{equation*}
A set $A$ is called exceptional, if ${\Capa}(A)=0$; see \cite[Theorem 3.4.4]{O}.
Then there exist a Borel exceptional
set ${\cal N}_0\subset M$ and a Hunt process
$X:=\{(X_t)_{t\geq 0}, (\Pp_x)_{x\in M\backslash {\cal N}_0}, (\F_t)_{t\geq 0} \}$
properly associated with  $(\eta,\F)$.
Namely, $X$ is a
quasi-left continuous
strong Markov process on $M\setminus {\cal N}_0$
with the quasi-left continuity on the time interval $(0,\infty)$ (see, e.g.,  \cite[Subsection 3.1]{O} for details),
and for any $\alpha>\alpha_0$ and bounded Borel function $f\in L^2(M;\mu)$,
\begin{equation*}
\Ee_x\left[\int_0^{\infty}e^{-\alpha t}f(X_t)\,\d t\right]=G_{\alpha}f(x), \quad \text{$\mu$-a.e.\ $x\in M$}
\end{equation*}
(see, e.g.,  \cite[Theorem 3.3.4]{O} or \cite[Theorem 4.1]{FU}).
Note that $X$ has c\`adl\`ag sample paths by definition,
and
the filtration $(\F_t)_{t\geq 0}$ can be
assumed
complete and right continuous
(see, e.g., \cite[Subsection 3.1]{O}).

For a Borel set $B\subset M$, let
\begin{equation*}
\sigma_{B}=\inf\left\{t>0: X_t\in B\right\},\quad  \hat{\sigma}_B=\inf\left\{t>0 : X_{t-}\in B\right\}.
\end{equation*}
A Borel set ${\cal N}\subset M$ is called properly exceptional,
if $\mu({\cal N})=0$ and
$\Pp_x(\sigma_{{\cal N}}=\hat{\sigma}_{{\cal N}}=\infty)=1$
for any $x\in M\setminus{\cal N}$.
We can take a properly exceptional set
${\cal N}$ so that ${\cal N}\supset {\cal N}_0$ (see, e.g., \cite[Theorem 4.2 (ii)]{FU}).
In particular, $X|_{M\setminus {\cal N}}:=\{(X_t)_{t\geq 0}, (\Pp_x)_{x\in M\backslash {\cal N}}, (\F_t)_{t\geq 0} \}$ is
still a Hunt process properly associated with $(\eta,\F)$.

\subsection{Martingale property and LIL of Hunt processes}
Let ${\rm diag}=\{(x,x): x\in \R^d\}$
be the diagonal set, and
let $J (x,y)$ be a non-negative Borel measurable function
on $\R^d\times \R^d\setminus{\rm diag}$ such that
\begin{equation}\label{eq:levy-type}
\sup_{x\in \R^d}\int_{\R^d\setminus\{x\}}
(1\wedge |x-y|^2 ) J_s(x,y)\,{\rm d}y<\infty
\end{equation}
and
\begin{equation}\label{eq:a-s}
\sup_{x\in \R^d}\int_{\{J_s(x,y) \ne 0\}}\frac{J_a(x,y)^2}{J_s(x,y)}\,{\rm d}y<\infty,
\end{equation}
where
$$J_s(x,y)=\frac{1}{2}(J(x,y)+J(y,x)), \quad J_a(x,y)=\frac{1}{2}(J(x,y)-J(y,x)).$$

For $n\in {\mathbb N}$ and $u, v \in C_c^{\infty}(\R^d)$, we define
$$L_nu(x)=\int_{\{|y-x|\geq 1/n\}}(u(y)-u(x))\,J(x,y)\,{\rm d}y$$
and
$$\eta_n(u,v)=-\int_{\R^d}L_n u(x)v(x)\,\d x.$$
Note that by \eqref{eq:levy-type},
the right hand side of the equality above is absolutely convergent.
Let
$${\cal D}(D)=\left\{u\in L^2(\R^d;\d x)
:\,\iint_{
\R^d\times \R^d\setminus{\rm diag}}(u(x)-u(y))^2J(x,y)\,{\rm d}x\,{\rm d}y<\infty\right\}$$
and
$$D(u,v)=\iint_{
\R^d\times \R^d\setminus{\rm diag}}(u(x)-u(y))(v(x)-v(y))J(x,y)\,{\rm d}x\,{\rm d}y.$$
Then, by \eqref{eq:levy-type} again, $(D,{\cal D}(D))$ is a symmetric Dirichlet form on $L^2(\R^d;\d x)$
such that $C_c^{\infty}(\R^d)\subset {\cal D}(D)$.
Let $\F$ be the closure of $C_c^{\infty}(\R^d)$
with respect to the norm $\|u\|:=\sqrt{D(u,u)+\|u\|_{L^2(\R^d;\d x)}^2}$.
Then $(D,\F)$ becomes a regular symmetric Dirichlet form on $L^2(\R^d;\d x)$.
Furthermore, according to  \cite[Theorem 2.1]{SW} (see also \cite[Proposition 2.1]{FU}),
the limit $$\eta(u,v):=\lim_{n\to\infty}\eta_n(u,v)$$ exists
for all $u,v\in C_c^{\infty}(\R^d)$ such that
$$\eta(u,v)=\frac{1}{2}D(u,v)+\iint_{
\R^d\times \R^d\setminus{\rm diag}}(u(x)-u(y))v(y)J_a(x,y)\,\d x\,\d y.$$
In particular,
$(\eta,\F)$ becomes a regular lower bounded semi-Dirichlet form on $L^2(\R^d;\d x)$.

 \ \

In what follows,
let $X:=\{(X_t)_{t\geq 0}, (\Pp_x)_{x\in \R^d\backslash {\cal N}_0}, (\F_t)_{t\geq 0}\}$
be a Hunt process on $\R^d$ properly associated with $(\eta,{\cal F})$,
where  ${\cal N}_0$ is
an exceptional set as
mentioned in Subsection \ref{subsect:semi-dirichlet}.
According to the Beurling-Deny type decomposition
for semi-Dirichlet forms (see \cite[Theorem 5.2.1]{O}),
there are no local part and no killing term
in the lower bounded semi-Dirichlet form $(\eta,\F)$ given above,
and so the associated process $X$ is also of pure-jump type.
In order to present sufficient conditions on the jumping kernel $J(x,y)$
such that $X$ itself is a
purely discontinuous martingale
with finite second moment,
we will make use full of the expression for the generator associated with $(\eta,\F)$.
For this purpose,
we impose the following assumption on $J(x,y)$.

\begin{assum}\label{assum:jump-semi}\it
The jumping kernel $J(x,y)$ satisfies
the next three conditions{\rm :}
\begin{itemize}
\item[{\rm(i)}]  for any $\varepsilon>0$, $x\in \R^d$ and $1\le i\le d$,
\begin{equation}\label{e:non-drift}
\int_{\{|x-y|\geq \varepsilon\}}(y-x)^{(i)}J(x,y)\,\d y=0;
\end{equation}

\item[{\rm(ii)}] $J(x,y)$ has the second moment in the sense that
\begin{equation}\label{jjj}
\sup_{x\in \R^d}\int
_{\R^d\setminus\{x\}}|x-y|^2 J(x,y)\,dy<\infty;
\end{equation}
\item [{\rm(iii)}] the function
$$x\mapsto \int_{\{|y-x|\geq 1\}}J(x,y)\,\d y $$
belongs to $L^2(\R^d;\d x)\cup L^{\infty}(\R^d;\d x)$.
\end{itemize}
\end{assum}

\begin{lem}\label{prop:semi-dirichlet-mtg}
Let $(L,{\cal D}(L))$ be the {\rm ($L^2$-)}generator of  $(\eta,\F)$.
Under Assumption {\rm \ref{assum:jump-semi}},
we have the following two statements.
\begin{itemize}
\item[{\rm(1)}] $C_c^{\infty}(\R^d)\subset {\cal D}(L)$, and for any $u\in C_c^{\infty}(\R^d)$,
\begin{equation}\label{e:ope1}
Lu(x)
=\int_{\R^d\setminus\{x\}}(u(y)-u(x)-\langle \nabla u(x), y-x\rangle) J(x,y)\,\d y.
\end{equation}
Moreover,
$(L,C_c^{\infty}(\R^d))$
extends to $C_b^2(\R^d)$,
and the expression above remains valid for any $u\in C_b^2(\R^d)$.
\item [{\rm(2)}]
There exists a Borel properly exceptional set
${\cal N}\supset {\cal N}_0$
such that
for
any $u\in C_b^2(\R^d)$,
$$
M_t^{[u]}=u(X_t)-u(X_0)-\int_0^t Lu(X_s)\,\d s,\quad t\geq 0,
$$
is a $\Pp_x$-martingale for each
$x\in \R^d\setminus{\cal N}$.
Moreover, $X|_{M\setminus {\cal N}}$ is conservative.
\end{itemize}
\end{lem}
\begin{proof}
According to \eqref{e:non-drift}, for any $n\ge1$,
$$L_nu(x)=\int_{\{|y-x|\geq 1/n\}}(u(y)-u(x)-\langle
\nabla
u(x),y-x\rangle)\,J(x,y)\,{\rm d}y.$$
Let $L$ be as in
\eqref{e:ope1}. It is obvious that, under
Assumption \ref{assum:jump-semi} (ii), $Lu$ is
pointwisely
well defined for any $u\in C_c^\infty(\R^d)$. Moreover,
\begin{align*}
|Lu(x)-L_nu(x)|=
&\left|\int_{\{|y-x|<1/n\}} (u(y)-u(x)-\langle
\nabla
u(x),y-x\rangle)\,J(x,y)\,{\rm d}y\right|\\
\le
&\|\nabla^2 u\|_\infty \sup_{x\in \R^d}\int_{\R^d\setminus\{x\}} |y-x|^2\,J(x,y)\,{\rm d}y<\infty.
\end{align*}
Then, by
Assumption \ref{assum:jump-semi} (ii) again and the dominated convergence theorem,
for any $f,g\in C_c^\infty(\R^d)$,
$$\lim_{n\to \infty}\eta_n(f,g)=-\lim_{n\to\infty}\int
_{\R^d}
L_nf(x)g(x)\,\d x=-\int
_{\R^d}
Lf(x)g(x)\,\d x.$$
In particular, the equality above shows that the operator $L$ is the generator of  $(\eta,\F)$.
Following the argument in step 2 of \cite[Theorem 2.2]{SW} and using \eqref{eq:levy-type} and
Assumption \ref{assum:jump-semi} (iii),
we know that $L$ maps $C_c^\infty(\R^d)$ into $L^2(\R^d;dx)$.
We also note that the operator $(L,C_c^{\infty}(\R^d))$ extends to $C_b^2(\R^d)$
in a similar way as in \cite[Section 5]{FU} or \cite[Theorem 2.37]{BSW}.
Hence we arrive at the assertion (1).

Applying \cite[Theorem 4.3]{FU} to $(L,C_b^2(\R^d))$,
we can obtain the assertion (2).
We note that, even though \cite[Theorem 4.3]{FU} requires
the continuity of $Lu$ for any
$u\in C_c^{\infty}(\R^d)$,
the proof of this theorem is still true without this assumption.
\end{proof}

According to \eqref{e:non-drift} and Lemma \ref{prop:semi-dirichlet-mtg},
we obtain the statement below
by letting $N(x,\d z):=J(x,x+z)\,\d z$,
and following the generator approach to Theorem \ref{thm:markov-mtg}
(see Subsection \ref{section2.1.1})
and the proof of Theorem \ref{thm:lil}.

\begin{thm}\label{thm:lil-1}
Let Assumption {\rm \ref{assum:jump-semi}} hold.
Then, we have
\begin{enumerate}
\item[{\rm (1)}]
$X$ is a
purely discontinuous martingale such that
for each $t>0$ and $i=1,\dots, d$, $X_t^{(i)}$ has finite second moment and
the quadratic variation of $X$
is given by
\begin{equation*}
\langle X^{(i)}, X^{(j)}\rangle_t
=\int_0^t \left(\int_{\R^d\setminus\{0\}}z^{(i)}z^{(j)}J(X_s,X_s+z)\,\d z\right) \,\d s ,
\quad 1\leq i,j\leq d,\,t>0.
\end{equation*}
\item[{\rm (2)}]
If the kernel $N(x,\d z):=J(x,x+z)\,\d z$ satisfies
Assumption {\rm \ref{assum:bound}},
then the assertion of Theorem {\rm \ref{thm:lil}} is valid
for every $x\in \R^d\setminus{\cal N}$.
\end{enumerate}
\end{thm}

\subsection{Examples}
In this subsection, we provide a class of jump-type Hunt processes
generated by regular lower bounded semi-Dirichlet forms
such that they are
purely discontinuous martingales with finite second moment,
and satisfy Khintchine's law of the iterated logarithm.
\begin{exam}\label{e:HUNT}\rm
Let $J(x,y)$ be a non-negative Borel function on
$\R^d\times\R^d\setminus{\rm diag}$
given by
\begin{equation}\label{eq:kernel-exam}
J(x,y)=\frac{c(x)}{|x-y|^{d+\alpha(|x-y|)}}
\end{equation}
such that the following two conditions hold:
\begin{enumerate}
\item[(i)]
$\alpha(r)$ is a positive function on
$(0,\infty)$
such that
\begin{equation}\label{eq:integrable-0}
\int_0^{\infty}  {r^{1-\alpha(r)}}\,\d r<\infty;
\end{equation}
\item[(ii)] $c(x)$ is a function on $\R^d$
bounded from below and above by positive constants, and
$$
\int_0^1 \frac{g_c(r)^2}{r^{1+\alpha(r)}}\,\d r<\infty,
$$ where $$
g_c(r)=\sup_{x,y\in \R^d: |x-y|=r}|c(x)-c(y)|.
$$
\end{enumerate}
Then, the jumping kernel $J(x,y)$ above generates
a regular lower bounded semi-Dirichlet form $(\eta,{\cal F})$
on $L^2(\R^d;\d x)$.
Indeed, by definition,
$$J_s(x,y)=\frac{1}{2}\frac{c(x)+c(y)}{|x-y|^{d+\alpha(|x-y|)}},
\quad J_a(x,y)=\frac{1}{2}\frac{c(x)-c(y)}{|x-y|^{d+\alpha(|x-y|)}}.$$
Then
\begin{align*}
\int_{\R^d}(1\wedge |x-y|^2)J_s(x,y)\,\d y
&=\frac{1}{2}\int_{\R^d}(1\wedge |x-y|^2)\frac{c(x)+c(y)}{|x-y|^{d+\alpha(|x-y|)}}\,\d y\\
&\leq c_1 \int_{\R^d}\frac{1\wedge |x-y|^2}{|x-y|^{d+\alpha(|x-y|)}}\,\d y
=c_2\int_0^{\infty}\frac{1\wedge r^2}{r^{1+\alpha(r)}}\,\d r,
\end{align*}
which implies \eqref{eq:levy-type}.
Since
\begin{align*}
\frac{J_a(x,y)^2}{J_s(x,y)}
=\frac{1}{2}\frac{(c(x)-c(y))^2}{c(x)+c(y)}\frac{1}{|x-y|^{d+\alpha(|x-y|)}}
\leq c_3\frac{(c(x)-c(y))^2}{|x-y|^{d+\alpha(|x-y|)}}
\leq c_4\frac{g_c(|x-y|)^2}{|x-y|^{d+\alpha(|x-y|)}},
\end{align*}
we also obtain \eqref{eq:a-s}.

It is obvious that
Assumption \ref{assum:jump-semi} (i) holds.
By \eqref{eq:integrable-0} and the calculations above,
one can see that
Assumption \ref{assum:jump-semi} (ii) and (iii) are satisfied.
Since the kernel $N(x,\d z)=J(x,x+z)\,\d z$ fulfills Assumption \ref{assum:bound},
the statement of Theorem \ref{thm:lil-1} holds
for a Hunt process $X$ generated by $(\eta,{\cal F})$.

\ \

The concrete example for $\alpha(r)$ and $c(x)$
satisfying
the conditions (i) and (ii) above is as follows.
Let $\alpha(r)$ be a locally bounded
and positive measurable function on $(0,\infty)$ such that
\begin{equation*}
\limsup_{r\rightarrow+0}\alpha(r)<2, \quad \liminf_{r\rightarrow\infty}\alpha(r)>2.
\end{equation*}
Let $c(x)$ be a Lipschitz continuous function on $\R^d$
bounded from below and above
by positive constants.

\end{exam}

\begin{rem}\rm
To the best of our knowledge,
it is unknown in the literature
whether the martingale problem is well-posed or not
for the operator $(L,C_c^\infty(\R^d))$
defined by \eqref{e:ope1}
with the jumping kernel $J(x,y)$ in \eqref{eq:kernel-exam}.
In particular, we do not know whether $(L,C_c^\infty(\R^d))$ can generate a Feller semigroup or not.

On the other hand, we can construct a Hunt process on $\R^d$
associated with the jumping kernel $J(x,y)$ by using the Dirichlet form theory.
The price is to take into consideration the exceptional set
restricting the initial points of the process.
\end{rem}

\ \

We further present examples of the jumping kernels $J(x,y)$
such that  the statement of Theorem \ref{thm:lil-1} is valid
for  the associated Hunt processes.
These examples can be regarded as variants of the jumping kernels
given in \cite[Subsection 6.2, (9) and (13)]{FKV15}.

\begin{exam}\rm
\begin{enumerate}
\item[(1)]
Let $A$ be a Borel set of $\R^d\setminus\{0\}$ with positive Lebesgue measure
such that
\begin{enumerate}
\item[(a)] $A=-A \ (:=\{x\in \R^d\setminus\{0\} \mid -x\in A\}$);
\item[(b)] for any $(x^{(1)},\dots, x^{(d)})\in A$ and for any permutation $\{i_1,\dots,i_d\}$ of $\{1,\dots, d\}$,
$(x^{(i_1)},\dots, x^{(i_d)})\in A$.
\end{enumerate}
Let $J(x,y)$ be a non-negative Borel function on
$\R^d\times\R^d\setminus{\rm diag}$
 given by
\begin{equation*}
J(x,y)=\frac{c(x)}{|x-y|^{d+\alpha(|x-y|)}}{\bf 1}_{\{y-x\in A\}}.
\end{equation*}
Suppose that the functions $\alpha(r)$ and $c(x)$ satisfy (i) and (ii) as these in Example \ref{e:HUNT}.
Then, in the same manner as for \eqref{eq:kernel-exam},
we can show that the statement of Theorem \ref{thm:lil-1} is true
for the Hunt process
generated by a lower bounded semi-Dirichlet form
with the jumping kernel $J(x,y)$.
\item[(2)]
Let $n\in {\mathbb N}$.
Let $A_i \ (1\leq i\leq n)$ be a Borel set of $\R^d\setminus\{0\}$ with positive Lebesgue measure
such that (a) and (b) in (1) are fulfilled.
For each $i\in \{1,\dots, n\}$,
let $\alpha_i(r)$ be a positive function on $(0,\infty)$
and $c_i(x)$ a function on $\R^d$
such that (i) and (ii) in Example \ref{e:HUNT} are satisfied.
Then the same consequence as in (1) is valid for the
jumping kernel
\begin{equation*}
J(x,y)=\sum_{i=1}^n\frac{c_i(x)}{|x-y|^{d+\alpha_i(|x-y|)}}{\bf 1}_{\{y-x\in A_i\}}.
\end{equation*}
\end{enumerate}
\end{exam}

\medskip

\noindent \textbf{Acknowledgements.}
The authors would like to thank
Professor Takashi Kumagai and Professor Masayoshi Takeda
for their valuable comments on the draft of this paper.
They are grateful to the referee and the associate editor
for their valuable comments and suggestions,
which improved the results and presentation of the manuscript.
The research of Yuichi Shiozawa is supported
in part by JSPS KAKENHI No.\ JP17K05299.
The research of Jian Wang is supported by the National Natural Science Foundation of China (No.\ 11831014),
the Program for Probability and Statistics:
Theory and Application (No.\ IRTL1704)
and the Program for Innovative Research Team in Science and Technology
in Fujian Province University (IRTSTFJ).

\address{
Yuichi Shiozawa\\
Department of Mathematics \\
Graduate School of Science\\
Osaka University \\
Toyonaka 560-0043,
Japan
}
{\texttt{shiozawa@math.sci.osaka-u.ac.jp}}

\address{
Jian Wang\\
College of Mathematics and Informatics\\
 Fujian Key Laboratory of Mathematical
Analysis and Applications (FJKLMAA)\\
Center for Applied Mathematics of Fujian Province (FJNU)\\
Fujian Normal University\\
350007 Fuzhou, P.R. China
}
{\texttt{jianwang@fjnu.edu.cn}}
\end{document}